\documentclass[
final
 , nomarks
]{dmtcs-episciences}

\usepackage[utf8]{inputenc}
\usepackage{amsthm}
\usepackage{amsmath}
\usepackage{amssymb}
\usepackage{tikz}
\usepackage{amscd}
\usepackage{amsfonts}
\usepackage{dsfont}
\usepackage{enumitem}
\usepackage{caption}
\usepackage{subcaption}
\usetikzlibrary{automata}

\newtheorem{theorem}{Theorem}
\newtheorem{lemma}{Lemma}[section]
\newtheorem{corollary}[lemma]{Corollary}
\newtheorem*{theorem*}{Theorem}

\theoremstyle{definition}
\newtheorem{definition}[lemma]{Definition}
\theoremstyle{remark}
\newtheorem{remark}[lemma]{Remark}
\newtheorem{example}[lemma]{Example}

\newcommand{\MR}[1]{}

\newcommand{\itemref}[1]{\ref{#1}}

\newcommand{\M}{\mathcal M}
\newcommand{\Expect}{\mathbb{E}}
\newcommand{\bigOh}{\mathcal O}

\newcommand{\bfzero}{\boldsymbol{0}}
\newcommand{\bfones}{\boldsymbol{1}}
\DeclareMathOperator{\sign}{sign}

\DeclareMathOperator{\grad}{grad}

\newcommand{\R}{\mathbb{R}}

\def\hyph{-\penalty0\hskip0pt\relax}

\author[Sara Kropf]{Sara Kropf\footnotetext{The author is supported by the Austrian Science Fund (FWF):
  P~24644-N26.}\footnotetext{Email-address:
    \texttt{sara.kropf@aau.at}}}
\title[Variance and Covariance of Simultaneous Outputs of a
  Markov Chain]{Variance and Covariance of Several Simultaneous Outputs of a
Markov Chain}
\affiliation{
  Institut f\"ur Mathematik, Alpen-Adria-Universit\"at
  Klagenfurt}

\received{2015-12-1}
\accepted{2016-6-17}

\begin{document}
\publicationdetails{18}{2016}{3}{11}{1341}
\maketitle
\begin{abstract}
The partial sum of the states of a Markov chain or more generally a Markov source
  is asymptotically normally distributed under suitable
  conditions. One of these conditions is that the variance is
  unbounded. A simple combinatorial characterization of Markov
  sources which satisfy this condition is given in terms of cycles of
  the underlying graph of the Markov chain. Also Markov sources with
  higher dimensional alphabets are considered.

  Furthermore, the case of an unbounded covariance between two
  coordinates of the Markov source is combinatorically characterized. If the covariance
  is bounded, then the two coordinates are asymptotically
  independent. 

  The results are illustrated by several examples, like the number of
  specific blocks in $0$-$1$-sequences and the Hamming weight of the
  width-$w$ non-adjacent form.
\end{abstract}

\section{Introduction}
We investigate the random vector defined as the $n$-th partial sum of a Markov
source over a higher dimensional alphabet. Under suitable conditions,
this random variable is asymptotically jointly normally
distributed. Its mean and variance-covariance
matrix is linear in the number of summands (cf.\ \cite[Theorem~2.22]{Drmota:2009:random}). On the one hand, these conditions include
irreducibility and aperiodicity of the underlying graph of the Markov chain, which can be checked
easily for a given Markov chain. On the other hand, we also have to
check that the variance-covariance matrix is regular, which requires
technical computations. In this article, we give
a simple combinatorial characterization of Markov sources whose corresponding
variance-covariance matrix is singular.

The covariance between two coordinates of this
random vector is also of interest: If it is bounded, then these
two coordinates are asymptotically independent because of the joint
normal distribution. We give a
combinatorial characterization of this case.

These characterizations are given in terms of subgraphs of the
underlying graph of the Markov chain: For the variance-covariance matrix,
we only have to consider all cycles. A regular variance-covariance
matrix will be proven to be equivalent to the linear independence of
certain functions of cycles of the underlying graph of the Markov chain.
 For
the characterization of an unbounded covariance, we have
to consider functional digraphs. This result is proven using an extension
of the Matrix-Tree Theorem in
\cite{Chaiken:1982:matrixtree,Moon:1994:matrixtree}.

As Markov sources are closely related to automata and transducers,
our results can also be used for the asymptotic analysis of sequences
which can be computed by transducers. This includes the Hamming weight
of many syntactically defined digit expansions as performed in
\cite{Grabner-Heuberger-Prodinger:2004:distr-results-pairs,Heuberger-Kropf-Prodinger:2015:output,Heuberger-Kropf:2013:analy,Grabner-Thuswaldner:2000:sum-of-digits-negative,Heigl-Heuberger:2012:analy-digit}. Furthermore,
occurrences of digits or subwords can also be computed by
transducers. Their variance (and covariance) is analyzed in
\cite{Grabner-Heuberger-Prodinger-Thuswaldner:2005:analy-linear,Avanzi-Heuberger-Prodinger:2006:scalar-multip-koblit-curves,Heuberger-Prodinger:2006:analy-alter,Bender-Kochman:1993:distr-subwor,Nicodeme-Salvy-Flajolet:2002:motif,Flajolet-Szpankowski-Vallee:2006:hidden-word-statis,Goldwurm-Radicioni:2007:averag-value}.

In \cite{Heuberger-Kropf-Wagner:2014:combin-charac}, the variance of the output of a transducer as well as the
covariance between the input and the output were analyzed. In this article, we consider the more general setting
of Markov chains. The proofs are similar as those in \cite{Heuberger-Kropf-Wagner:2014:combin-charac}, but the results are valid
in a broader context and can be formulated more clearly. 
In contrast to \cite{Heuberger-Kropf-Wagner:2014:combin-charac}, we allow the input sequence of the
transducer to be generated by a Markov source. This allows us to model an input sequence for a transducer whose letters do not occur with equal
probabilities and/or have dependencies between the letters. The precise
relation between the setting of this article and that
of \cite{Heuberger-Kropf-Wagner:2014:combin-charac} is given in
Section~\ref{sec:results}.

As an example, we prove that the Hamming weight of the so-called width-$w$ non-adjacent
form is asymptotically jointly normally distributed for two different
values of $w\geq 2$. The width-$w$ non-adjacent form is a binary digit
expansion with digits in $\{0,\allowbreak\pm1,\allowbreak\pm3,\allowbreak\ldots,\allowbreak\pm(2^{w-1}-1)\}$ and
the syntactical rule that at most one of any $w$ adjacent digits is
non-zero. This digit expansion exists and is unique for every
integer (cf.\ \cite{muirstinson:minimality,avanzi:mywnaf}). Furthermore, it has minimal
Hamming weight among all digit expansions with this base and digit set.

The outline of this article is as follows: In Section~\ref{sec:preliminaries-indep-2},
we define our setting and the types of graphs we use to state the
combinatorial characterization of independent output sums and singular
variance-covariance matrices. These characterizations are given in
Section~\ref{sec:results} and examples are given in Section~\ref{sec:examples}. In
Section~\ref{sec:proofs}, we finally prove the results of Section~\ref{sec:results}.

\section{Preliminaries}\label{sec:preliminaries-indep-2}
In this article, a \emph{finite Markov chain} consists of a finite state space
$\{1,\ldots,M\}$, a finite set of transitions $\mathcal E$ between the states,
each with a positive transition probability, and a unique\footnote{This is no restriction as we can always add an additional state
  and the transitions starting in this state with probabilities corresponding
  to the non-degenerate initial distribution. The output functions are then
  extended by mapping these transitions to~$0$.} initial state
$1$. We
denote the transition probability for a transition $e$ by $p_{e}$. Then we
have
\begin{equation*}
  \sum_{\substack{e\in\mathcal E\\e\text{ starts in }i}}p_{e}=1
\end{equation*}for all states $i$. Note that
for all transitions $e\in\mathcal E$, we require $p_{e}>0$. Further note that
there may be multiple transitions between two states but always only a finite
number of them. This may be useful for different outputs later on.

The transition probabilities induce a probability
distribution on the paths of length $n$ starting in the initial state $1$. Let
$X_{n}$ be a random path of length $n$ according to this model. 

All states
of the underlying digraph of the Markov chain are
assumed to be
accessible from the initial state. Contracting each strongly connected component of
the underlying digraph gives an acyclic digraph, the so-called
condensation. We assume that this condensation has only one leaf (i.e., one
vertex with out-degree $0$). The strongly connected component corresponding to
this leaf is called \emph{final component}. We assume that the period (i.e.,
the greatest common divisor of the lengths of all cycles) of this final component is $1$. We call such
Markov chains \emph{finally connected} and \emph{finally aperiodic}.

Additionally we use \emph{output functions} $k\colon \mathcal E\to \mathbb
R$. The corresponding
random variable $K_{n}$ is the sum of all values of $k$ along a random path
$X_{n}$. We call $K_{n}$ the \emph{output sum} of the Markov chain with
respect to $k$.
We use several output functions $k_{1}$, \dots, $k_{m}$ and the corresponding
random variables $K_{n}^{(1)}$, \dots, $K_{n}^{(m)}$ simultaneously for one
Markov chain.

\begin{remark}
  Usually, one is interested in a function evaluated at the sequence of random states of the Markov
  chain. This is equivalent to this setting with an output function of the
  transitions: For the one direction, the restriction of the output function to the
  outgoing transitions of one state is constant for every state. For the other
  direction, we use the standard construction of the Markov chain with state
  space $\{(i,j)\mid 1\leq i,j\leq M\}$.

  Thus, our setting can be seen as a Markov source with a finite set of
  $m$-dimensional vectors as alphabet.
\end{remark}

We are interested in the joint distribution of the random
variables $K_{n}^{(1)}$, \dots, $K_{n}^{(m)}$. For one coordinate, we will prove that the expected
value of $K_{n}^{(i)}$ is $e_{i}n+\bigOh(1)$ for constants $e_{i}$. The
variance-covariance matrix of $K_{n}^{(1)}$, \dots, $K_{n}^{(m)}$ will turn
out to be $\Sigma
n+\bigOh(1)$ for a matrix $\Sigma$. We call $\Sigma$ the asymptotic variance-covariance matrix and its
entries the asymptotic variances and covariances. 

We will combinatorically
characterize Markov chains with output functions such that the
variance-covariance matrix is regular. Furthermore, we give a combinatorial
characterization of the case that the asymptotic covariance is zero. As this
is only influenced by two output functions, we restrict ourselves to
$K_{n}^{(1)}$ and $K_{n}^{(2)}$ in this case. 

\begin{remark}\label{remark:transducer}
  Markov chains with output functions are closely related to
  transducers with a probability distribution for the input: A transducer is defined to consist of a finite set of states, an initial
state, a set of final states, an input alphabet, an output alphabet and
 a finite set of transitions, where a transition starts in one state, leads to
 another state and has an input and an output label from the corresponding
 alphabets. See \cite[Chapter~1]{Berthe-Rigo:2010:combin} for a more formal definition. An example of a transducer is given in
 Figure~\ref{fig:first-ex-intro}. We label the transitions with ``input label
 $\mid$ output label''. The initial state is marked by an ingoing arrow
 starting at no other state and the final states are marked by 
 outgoing arrows leading to no other state.

\begin{figure}
  \centering
  \begin{tikzpicture}[auto, initial text=, >=latex, accepting text=,
  accepting/.style=accepting by arrow, every state/.style={minimum
    size=1.3em}]
\node[state, accepting, accepting where=below, initial, initial where=right]
(v0) at (1.300000, 0.000000) {};
\node[state, accepting, accepting where=below] (v1) at (-1.300000, 0.000000)
{};
\path[->] (v0.190.00) edge node[rotate=360.00, anchor=north] {$1\mid 0$} (v1.350.00);
\path[->] (v1.10.00) edge node[rotate=0.00, anchor=south] {$0\mid 1$} (v0.170.00);
\path[->] (v0) edge[loop above] node {$0\mid 1$} ();
\path[->] (v1) edge[loop above] node {$1\mid 1$} ();
\end{tikzpicture}
  \caption{A small example of a transducer.}
  \label{fig:first-ex-intro}
\end{figure}
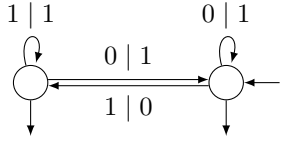

  A Markov chain with one output function can be obtained by a transducer with
  additional probability distributions for the outgoing transitions of each
  state and by deleting the input labels of the transducer. 

If we have two transducers where only the outputs of the
  transitions are different, we can choose probability distributions for the
  outgoing transitions of each state. Then we obtain a Markov chain with two
  output functions.
Thus, we can use our results for two
  output functions (see
  Examples~\ref{example:blocks} and~\ref{example:00-11}).
\end{remark}

\begin{remark}\label{remark:final-outputs}
We can additionally have \emph{final output functions} $f\colon\{1,\ldots, M\}\to \R$
for each output function $k$ and redefine
the random variable $K_{n}$ as the sum of the values of the output
function $k$ along a random path $X_{n}$ plus the final output $f$ of the final
state of this path. We will see that this does not change the main terms of the
asymptotic behavior. Thus, the results in Section~\ref{sec:results} are still
valid (see also Remark~\ref{remark:final-output-not-important}).  
\end{remark}

\begin{remark}\label{remark:exit-weights}
  The Parry measure are probabilities $p_{e}$ such that every path of length $n$
  has the same weight up to a constant factor
  (cf.\ \cite{Shannon:1948:mathem-theor-commun,Parry:1964:intrin-markov}). If
  we are interested in  probabilities such that every path of length $n$
  starting in the initial state $1$ has
  exactly the same weight, we have to use the Parry measure with additional
  \emph{exit weights}: Each path is additionally weighted by these exit weights according to the
  final state of the path (cf.\
  \cite[Lemma~4.1]{Heuberger-Kropf-Prodinger:ta:analy-carries}). 

However, the
  sum of the weights of all paths of length $n$ is no longer normalized: It
  differs from $1$ by an exponentially small error term for $n\to\infty$.  This gives an approximate
  equidistribution of all paths of length $n$. As we are interested in the
  asymptotic behavior for $n\to\infty$, the expected value and the variance of
  the corresponding measurable function $K_{n}$ can still be defined as
  usual.

  If we use these exit weights $w_{s}$ in our setting, the main terms of the
  asymptotic behavior are not changed. Thus, the theorems in
  Section~\ref{sec:results} are still valid (see also
  Remark~\ref{remark:final-output-not-important}).

  These exit weights can also be used to simulate final and non-final states of
  a transducer by setting the weights of non-final states to $0$. However, not
  all exit weights of the final component are allowed to be zero.
\end{remark}

Next, we define some subgraphs of the underlying graph of the final
component and extend the probabilities and the output functions to these subgraphs.

\begin{definition}We define the following types of directed graphs as subgraphs
  of the final component of the Markov chain.
\begin{itemize}
\item A \emph{rooted tree} is a weakly connected digraph with one vertex which has out-degree
$0$, while all other vertices have out-degree $1$. The vertex with out-degree $0$
is called the \emph{root} of the tree.

\item A \emph{functional digraph} is a digraph whose vertices have out-degree
$1$. Each component of a
functional digraph consists of a directed cycle and some trees rooted at
vertices of the cycle. For a functional digraph $D$, let $\mathcal C_{D}$ be the
set of all cycles of $D$.\end{itemize}
\end{definition}

The probabilities $p_{e}$ can be multiplicatively extended to a weight function for arbitrary subgraphs of the
Markov chain: Let $D$ be any subgraph of the underlying graph
of the Markov chain, then define the weight of $D$ by
\begin{equation*}
  p_{D}=\prod_{e\in D}p_{e}.
\end{equation*}
For a path $P$ of length $n$, this is exactly the probability $\mathbb P(X_{n}=P)$.

However, the output function $k$ is additively extended to cycles $C$ of the
underlying graph of the Markov chain by
\begin{equation*}
  k(C)=\sum_{e\in C}k(e).
\end{equation*}
This can further be extended to functional digraphs:
\begin{definition}\label{def:sum-over-graphs-2}
Let $\mathcal D_{1}$ and $\mathcal D_{2}$ be the sets of all spanning
subgraphs of the final component of the
Markov chain $\M$ which are functional digraphs and have one and two
components, respectively.

For functions $g$ and $h\colon\mathcal E\to \mathbb R$, we define
\begin{align*}
g(\mathcal D_{1})&=\sum_{D\in\mathcal D_{1}}p_D\sum_{C\in\mathcal
  C_{D}}g(C),\\
(g,h)(\mathcal D_{1})&=\sum_{D\in\mathcal D_{1}}p_{D}\sum_{C\in\mathcal
C_{D}}g(C)h(C),\\
(g,h)(\mathcal D_{2})&=\sum_{D\in\mathcal D_{2}}p_{D}\sum_{C_{1}\in\mathcal
C_{D}}\sum_{\substack{C_{2}\in\mathcal C_{D}\\C_{2}\neq
  C_{1}}}g(C_{1})h(C_{2}).
\end{align*}

As functions $g$ and $h$, we use the output functions $k_{1}$, \dots, $k_{m}$
and the constant function $\mathds 1(e)=1$.
\end{definition}
\section{Main Results}\label{sec:results}
In this section, we present the combinatorial characterization of output
functions of Markov chains which are asymptotically independent and
of Markov chains with output functions with a singular
variance-covariance matrix. The proofs can be found in Section~\ref{sec:proofs}.

If the underlying directed graph of the Markov chain is
$j$-regular, every transition has probability $1/j$, we only have two
output functions and the first output function
$k_{1}\colon\mathcal E\to\{0,1,\ldots,j-1\}$ is such that the restrictions of $k_{1}$ to the outgoing
transitions of one state is bijective for every state, then these results are
stated in \cite{Heuberger-Kropf-Wagner:2014:combin-charac} (see also Remark~\ref{remark:transducer}).

The next definition describes a sequence of random variables whose difference from its expected value is bounded for all elements.
\begin{definition}
  The output sum $K_{n}$ of a Markov chain is called
\emph{quasi\hyph deter\-min\-is\-tic} if there is a constant $a\in\R$ such that 
\begin{equation*}K_{n}=an+\bigOh(1)\end{equation*}
holds for all $n$. 
\end{definition}
Next we give the combinatorial characterization of output sums with bounded
variance in the case of a not necessarily independent identically distributed
input sequence.
\begin{theorem}\label{theorem:bounded-variance}
  For a finite, finally connected and finally aperiodic Markov chain $\M$ with an output function $k$, the following assertions are equivalent:
  \begin{enumerate}[label=(\alph*)]
  \item\label{item:var0-2} The asymptotic variance $v$ of the output sum is
    $0$.
\item \label{item:somewalks-2}There exists a state $s$ of the final component
  and a constant $a\in\R$ such that
  \begin{equation*}
    k(C)=a\mathds 1(C)
  \end{equation*}
holds for every closed walk $C$ of the final component visiting the state $s$
exactly once.
\item\label{item:allcycles-2}There exists a constant $a\in\R$ such that
  \begin{equation*}
    k(C)=a\mathds 1(C)
  \end{equation*}
holds for every directed cycle $C$ of the final component of  $\M$.
  \end{enumerate}

In that case, $an+\bigOh(1)$ is the expected value of the output sum and
Statement~\itemref{item:somewalks-2} holds for all states $s$ of the final
component.

If $\M$ is furthermore strongly connected, the following assertion is also
equivalent:
\begin{enumerate}[label=(\alph*)]\setcounter{enumi}{3}
\item\label{item:quasi-det-2}The random variable $K_{n}$ is
  quasi-deterministic with constant $a$.
\end{enumerate}
\end{theorem}

In the case that the value of the output function is $0$ or $1$ for each transition, there are only two trivial
output functions with asymptotic variance zero.
\begin{corollary}\label{cor:01output-2}
Let $k\colon\mathcal E\to\{0,1\}$. Then the asymptotic variance $v$ is zero
if and only if the
output function $k$ is constant on the final component.
\end{corollary}

The next theorem extends Theorem~\ref{theorem:bounded-variance} to the joint
distribution of several simultaneous output sums by combinatorically describing the case of a singular
variance-covariance matrix.
\begin{theorem}\label{theorem:sing-matrix}
  Let $\M$ be a finite, finally connected, finally aperiodic
  Markov chain with $m$ output functions $k_{1}$, \dots, $k_{m}$. Then the variance-covariance matrix $\Sigma$ is regular if and
  only if the functions $\mathds 1$, $k_{1}$, \dots, $k_{m}$ are linearly independent as
  functions from
  the vector space of cycles of the final component to the real numbers, i.e.\ there do not exist real
  constants $a_{0}$, \dots, $a_{m}$, not all
  zero, such that
  \begin{equation}\label{eq:linear-dependence-cycles}
    a_{0}\mathds 1(C)+a_{1}k_{1}(C)+\cdots+a_{m}k_{m}(C)=0
  \end{equation}
  holds for all cycles (or equivalently, for all closed walks) $C$ of the final component.

The random variables $K_{n}^{(1)}$, \dots,
  $K_{n}^{(m)}$ are asymptotically jointly normally distributed
  if and only if $\Sigma$ is regular.
\end{theorem}

\begin{remark}
  Theorems~\ref{theorem:bounded-variance} and~\ref{theorem:sing-matrix}
  and Corollary~\ref{cor:01output-2} are independent of the choice of the
  probabilities of the transitions. Only the structure of the underlying graph
  of the Markov chain and the output functions influence the result. Note,
  however, that according to our general assumptions, all transitions have
  \emph{positive} probability.
\end{remark}

The next theorem gives a combinatorial characterization of output functions of
a Markov chain which are asymptotically independent. As this characterization
is given by the covariance, we can restrict ourselves to two output functions
without loss of generality.
\begin{theorem}\label{thm:comb-2}Let $\M$ be a finite, finally
  connected, finally aperiodic Markov chain with two output functions $k_{1}$ and $k_{2}$. 
  
  Then the random variable $K_{n}^{(i)}$ has the expected
  value $e_{i}n+\bigOh(1)$ and the variance $v_{i}n+\bigOh(1)$ where the constants are
  \begin{align}\label{eq:expected-value-indep}
    e_{i}&=\frac{k_{i}(\mathcal D_1)}{\mathds 1(\mathcal D_1)},\\
v_{i}&=\frac{1}{\mathds 1(\mathcal D_{1})}\big((k_{i}-e_{i}\mathds
1,k_{i}-e_{i}\mathds 1)(\mathcal D_{1})-(k_{i}-e_{i}\mathds 1,k_{i}-e_{i}\mathds 1)(\mathcal D_{2})\big)\nonumber
  \end{align}
for $i=1$, $2$.

The
 covariance of $K_{n}^{(1)}$ and $K_{n}^{(2)}$ is  $cn+\bigOh(1)$ with the constant
\begin{equation*}
c=\frac{1}{\mathds 1(\mathcal D_{1})}\big((k_{1}-e_{1}\mathds
1,k_{2}-e_{2}\mathds 1)(\mathcal D_{1})-(k_{1}-e_{1}\mathds 1,k_{2}-e_{2}\mathds 1)(\mathcal D_{2})\big).
\end{equation*}

  The random variables $K^{(1)}_{n}$ and $K^{(2)}_{n}$ are asymptotically independent if and only if
\begin{equation*}
(k_{1}-e_{1}\mathds
1,k_{2}-e_{2}\mathds 1)(\mathcal D_{1})=(k_{1}-e_{1}\mathds 1,k_{2}-e_{2}\mathds 1)(\mathcal D_{2}).
\end{equation*}
\end{theorem}

In the case that the expected values of $K^{(1)}_{n}$ and $K^{(2)}_{n}$ are
both bounded, i.e.\ $e_{1}=e_{2}=0$,
these random variables are asymptotically independent if and only if
\begin{equation*}
  (k_{1},k_{2})(\mathcal D_{1})=(k_{1},k_{2})(\mathcal D_{2}).
\end{equation*}

\section{Examples}\label{sec:examples}
In this section, we first prove the asymptotic joint normal
distribution of the Hamming weights of two different digit expansions by using
Theorem~\ref{theorem:sing-matrix}. Then we investigate the independence of
length $2$ blocks of $0$-$1$-sequences by using
Theorem~\ref{thm:comb-2}. In both cases we start with two transducers to
construct a Markov chain with two output functions, once as a
Cartesian product, once via Remark~\ref{remark:transducer}.
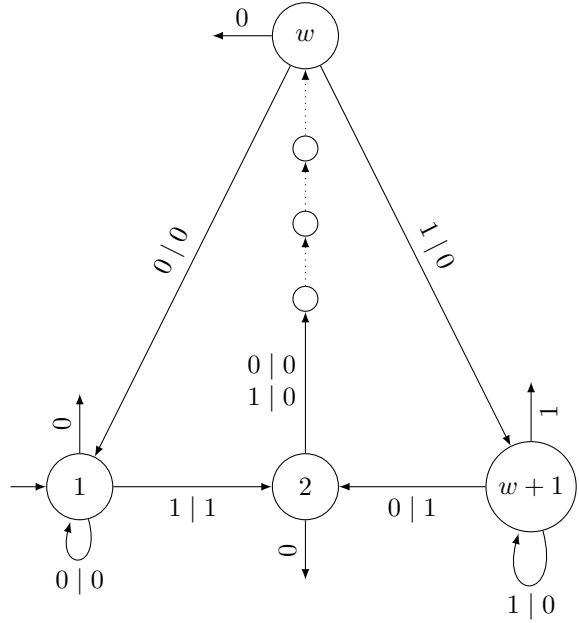
\begin{figure}
\centering
\newcommand{\Bold}[1]{\mathbf{#1}}\begin{tikzpicture}[auto, initial text=, >=latex, accepting text=, accepting/.style=accepting by arrow]
\node[state, initial] (v0) at (0.000000,
0.000000) {$1$};
\path[->] (v0.90.00) edge node[rotate=270.00, anchor=north] {$0$} ++(90.00:5ex);
\node[state] (v1) at (3.000000, 0.000000) {$2$};
\path[->] (v1.270.00) edge node[rotate=450.00, anchor=south] {$0$} ++(270.00:5ex);
\node[state] (v2) at (6.000000, 0.000000) {$w+1$};
\path[->] (v2.90.00) edge node[rotate=90.00, anchor=north] {$1$} ++(90.00:5ex);
\node[state] (v3) at (3, 6) {$w$};
\path[->] (v3.180.00) edge node[rotate=0.00, anchor=south] {$0$}
++(180.00:5ex);
\node[state, minimum size=0em] (v4) at (3, 2.5) {};
\node[state, minimum size=0em] (v5) at (3, 3.5) {};
\node[state, minimum size=0em] (v6) at (3, 4.5) {};
\path[->] (v0) edge node[rotate=0.00, anchor=north] {$1\mid 1$} (v1);
\path[->] (v1) edge node[auto, align=center] {$0\mid
  0$\\$1\mid 0$} (v4);
\path[->] (v4) edge[dotted] node {} (v5);
\path[->] (v5) edge[dotted] node {} (v6);
\path[->] (v6) edge[dotted] node {} (v3);
\path[->] (v0) edge[loop below] node {$0\mid 0$} ();
\path[->] (v2) edge node[rotate=360.00, anchor=north] {$0\mid 1$} (v1);
\path[->] (v2) edge[loop below] node {$1\mid 0$} ();
\path[->] (v3) edge node[rotate=63.43, anchor=south] {$0\mid 0$} (v0);
\path[->] (v3) edge node[rotate=-63.43, anchor=south] {$1\mid 0$} (v2);
\end{tikzpicture}
\caption{Transducer $\mathcal T(w)$ to compute the Hamming weight of the \mbox{width-$w$} non-adjacent form.}
\label{fig:wNAF-indep-2}
\end{figure}
\begin{example}[Width-$w$ non-adjacent forms]\label{example:w-naf}
Let $2\leq w_{1}< w_{2}$ be integers. We consider the asymptotic joint
distribution of the Hamming weight of the width-$w_{1}$ non-adjacent form ($w_{1}$-NAF) and the Hamming weight
of the $w_{2}$-NAF. The width-$w$ non-adjacent form is a binary digit
expansion with digit set $\{0,\pm1,\pm3,\ldots,\pm(2^{w-1}-1)\}$ and
the syntactical rule that at most one of any $w$ adjacent digits is non-zero.

 It will turn out that this distribution is normal if and only if the
variance-covariance matrix is regular. Using 
Theorem~\ref{theorem:sing-matrix}, we have to find closed
walks in the corresponding Markov chain such that
all coefficients in \eqref{eq:linear-dependence-cycles} have to be zero.

The transducer $\mathcal T(w)$ in Figure~\ref{fig:wNAF-indep-2} computes the Hamming weight of
the $w$-NAF of the integer $n$ when the input is the binary expansion of
$n$ (cf.\ \cite{Heuberger-Kropf:2013:analy}). It has $w+1$ states. Next, we construct the Cartesian product of the transducers for $w_{1}$ and
$w_{2}$ and choose any non-degenerate probability distribution, i.e.\ with all
  probabilities non-zero, for the outgoing
transitions of a state. Thus, we obtain a Markov chain $\M$ with
$(w_{1}+1)(w_{2}+1)$ states with two different
output functions $h_{1}$ and $h_{2}$ corresponding to the outputs of the transducers for $w_{1}$ and
$w_{2}$, respectively. We can now use Theorem~\ref{theorem:sing-matrix} to prove
that these two Hamming weights are asymptotically jointly normally distributed.

The Cartesian product of two closed walks in $\mathcal T(w_{1})$ and $\mathcal
T(w_{2})$ with the same input sequence is a closed walk in $\M$. We construct three
different closed walks and prove that all three coefficients in
\eqref{eq:linear-dependence-cycles} have to be zero. For brevity, we denote a closed walk
in the Cartesian product $\mathcal M$ and its projections to $\mathcal T(w_{1})$ and
$\mathcal T(w_{2})$ by the same letter.

First, we choose the closed
walk $C_{1}$ starting in state $1$ with input sequence $0$. We obtain
$h_{1}(C_{1})=0$ in $\mathcal T(w_{1})$, $h_{2}(C_{1})=0$ in $\mathcal T(w_{2})$ and $\mathds 1(C_{1})=1$. Second, we choose the
closed walk $C_{2}$ starting in $1$ with input sequence
$10^{w_{2}-1}$. Because $w_{1}<w_{2}$ and the loop at state $1$, $C_{2}$
is a closed walk in $\mathcal T(w_{1})$ and $\mathcal T(w_{2})$. We obtain
$h_{1}(C_{2})=1$ in $\mathcal T(w_{1})$, $h_{2}(C_{2})=1$ in $\mathcal T(w_{2})$ and $\mathds 1(C_{2})=w_{2}$. The third choice
depends on whether $w_{1}= w_{2}-1$ or not:
\begin{itemize}
\item $w_{1}\neq w_{2}-1$: We choose the closed walk $C_{3}$ starting in $1$ with input sequence
  $10^{w_{1}-1}10^{w_{1}-1}0^{\alpha}$ where
  $\alpha=\max(w_{2}-2w_{1},0)$. On the one hand, this is a closed walk in $\mathcal T(w_{1})$
  consisting of two times the cycle $1\to w_{1}\to1$ and $\alpha$ times the loop
  at state $1$. On the other hand, this is a closed walk in $\mathcal
  T(w_{2})$ consisting of the cycle $1\to w_{2}\to 1$ and the correct number of
  loops at state $1$. We
  obtain $h_{1}(C_{3})=2$ in $\mathcal T(w_{1})$, $h_{2}(C_{3})=1$ in
  $\mathcal T(w_{2})$ and $\mathds 1(C_{3})=\max(w_{2},
  2w_{1})$.
\item $w_{1}=w_{2}-1$: We choose the closed walk $C_{3}$ starting in $1$ with input sequence
  $10^{w_{1}-1}10^{w_{1}-1}10^{w_{1}-1}$. On the one hand, this is a closed
  walk in $\mathcal T(w_{1})$ consisting of three times the cycle $1\to
  w_{1}\to1$. On the other hand, this is a closed walk in $\mathcal T(w_{2})$
  consisting of the closed walk $1\to w_{2}\to w_{2}+1\to w_{2}\to 1$ and the correct number of
  loops at state $1$. We obtain $h_{1}(C_{3})=3$ in
  $\mathcal T(w_{1})$,
  $h_{2}(C_{3})=2$ in $\mathcal T(w_{2})$ and $\mathds 1(C_{3})=3w_{1}$.
\end{itemize}

This yields a system of linear equations for the coefficients $a_{0}$,
$a_{1}$ and $a_{2}$ with coefficient matrix
\begin{equation*}
  \begin{pmatrix}
    1&0&0\\
    w_{2}&1&1\\
    \max(w_{2},2w_{1})&2&1
  \end{pmatrix}\quad\text{ or }\quad
  \begin{pmatrix}
    1&0&0\\
    w_{2}&1&1\\
    3w_{1}&3&2
  \end{pmatrix},
\end{equation*}
 which only has the trivial solution. Thus, the
Hamming weights of the $w_{1}$-NAF and the $w_{2}$-NAF are asymptotically
jointly normally distributed, independently of the choice of the
distributions for the Markov chain.
\end{example}

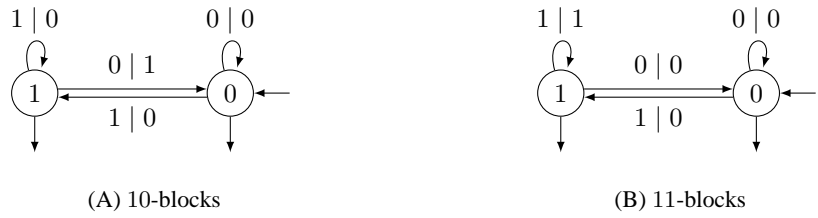
\begin{figure}
\centering
\begin{subfigure}[b]{0.45\textwidth}\centering
\newcommand{\Bold}[1]{\mathbf{#1}}\begin{tikzpicture}[auto, initial text=,
  >=latex, accepting text=, accepting/.style=accepting by arrow,  every state/.style={minimum size=1.3em}]
\node[state, initial, initial where=right, accepting, accepting where=below] (v0) at (1.300000, 0.000000) {$0$};
\node[state, accepting, accepting where=below] (v1) at (-1.300000, 0.000000) {$1$};
\path[->] (v1.10.00) edge node[rotate=0.00, anchor=south] {$0\mid 1$} (v0.170.00);
\path[->] (v0) edge[loop above] node {$0\mid 0$} ();
\path[->] (v0.190.00) edge node[rotate=360.00, anchor=north] {$1\mid 0$} (v1.350.00);
\path[->] (v1) edge[loop above] node {$1\mid 0$} ();
\end{tikzpicture}
\caption{$10$-blocks}
\end{subfigure}\quad
\begin{subfigure}[b]{0.45\textwidth}\centering

\newcommand{\Bold}[1]{\mathbf{#1}}\begin{tikzpicture}[auto, initial text=,
  >=latex, accepting text=, accepting/.style=accepting by arrow, every state/.style={minimum size=1.3em}]
\node[state, initial, initial where=right, accepting, accepting where=below] (v0) at (1.300000, 0.000000) {$0$};
\node[state, accepting, accepting where=below] (v1) at (-1.300000, 0.000000) {$1$};
\path[->] (v1.10.00) edge node[rotate=0.00, anchor=south] {$0\mid 0$} (v0.170.00);
\path[->] (v0) edge[loop above] node {$0\mid 0$} ();
\path[->] (v0.190.00) edge node[rotate=360.00, anchor=north] {$1\mid 0$} (v1.350.00);
\path[->] (v1) edge[loop above] node {$1\mid 1$} ();
\end{tikzpicture}
\caption{$11$-blocks}
\end{subfigure}
\caption{Transducers to compute the number of $10$- and $11$-blocks.}
\label{fig:11-indep-2}
\end{figure}
\begin{figure}
\centering
\begin{subfigure}[b]{\textwidth}\centering
\newcommand{\Bold}[1]{\mathbf{#1}}
\begin{tikzpicture}[auto, initial text=, >=latex, accepting text=,
  accepting/.style=accepting by arrow, every state/.style={minimum
    size=1.3em}]
\useasboundingbox (-1.65, -0.65) rectangle (1.65, 1.5);
\node[state] (v0) at (1.200000, 0.000000) {$0$};
\node[state] (v1) at (-1.200000, 0.000000) {$1$};
\path[->, line width=0.8pt] (v0.190.00) edge node[rotate=360.00, anchor=north] {} (v1.350.00);
\path[->, color=gray] (v1.10.00) edge node[rotate=0.00, anchor=south] {} (v0.170.00);
\path[->, color=gray] (v0) edge[loop above] node {} ();
\path[->, line width=0.8pt] (v1) edge[loop above] node {} ();
\end{tikzpicture}\quad
\begin{tikzpicture}[auto, initial text=, >=latex, accepting text=,
  accepting/.style=accepting by arrow, every state/.style={minimum
    size=1.3em}]
\useasboundingbox (-1.65, -0.65) rectangle (1.65, 1.5);
\node[state] (v0) at (1.200000, 0.000000) {$0$};
\node[state] (v1) at (-1.200000, 0.000000) {$1$};
\path[->, line width=0.8pt] (v0.190.00) edge node[rotate=360.00, anchor=north] {} (v1.350.00);
\path[->, line width=0.8pt] (v1.10.00) edge node[rotate=0.00, anchor=south] {} (v0.170.00);
\path[->, color=gray] (v0) edge[loop above] node {} ();
\path[->, color=gray] (v1) edge[loop above] node {} ();
\end{tikzpicture}\quad
\begin{tikzpicture}[auto, initial text=, >=latex, accepting text=,
  accepting/.style=accepting by arrow, every state/.style={minimum
    size=1.3em}]
\useasboundingbox (-1.65, -0.65) rectangle (1.65, 1.5);
\node[state] (v0) at (1.200000, 0.000000) {$0$};
\node[state] (v1) at (-1.200000, 0.000000) {$1$};
\path[->, color=gray] (v0.190.00) edge node[rotate=360.00, anchor=north] {} (v1.350.00);
\path[->, line width=0.8pt] (v1.10.00) edge node[rotate=0.00, anchor=south] {} (v0.170.00);
\path[->, line width=0.8pt] (v0) edge[loop above] node {} ();
\path[->, color=gray] (v1) edge[loop above] node {} ();
\end{tikzpicture}
\caption{$\mathcal D_{1}$}
\end{subfigure}\\
\begin{subfigure}[b]{\textwidth}\centering
\newcommand{\Bold}[1]{\mathbf{#1}}\begin{tikzpicture}[auto, initial text=, >=latex, accepting text=, accepting/.style=accepting by arrow, every state/.style={minimum size=1.3em}]
\node[state] (v0) at (1.200000, 0.000000) {$0$};
\node[state] (v1) at (-1.200000, 0.000000) {$1$};
\path[->, color=gray] (v0.190.00) edge node[rotate=360.00, anchor=north] {} (v1.350.00);
\path[->, color=gray] (v1.10.00) edge node[rotate=0.00, anchor=south] {} (v0.170.00);
\path[->, line width=0.8pt] (v0) edge[loop above] node {} ();
\path[->, line width=0.8pt] (v1) edge[loop above] node {} ();
\end{tikzpicture}
\caption{$\mathcal D_{2}$}
\end{subfigure}
\caption{Functional digraphs of the transducers of
  Examples~\ref{example:blocks} and \ref{example:00-11}.}
\label{fig:func-digraph-2}
\end{figure}
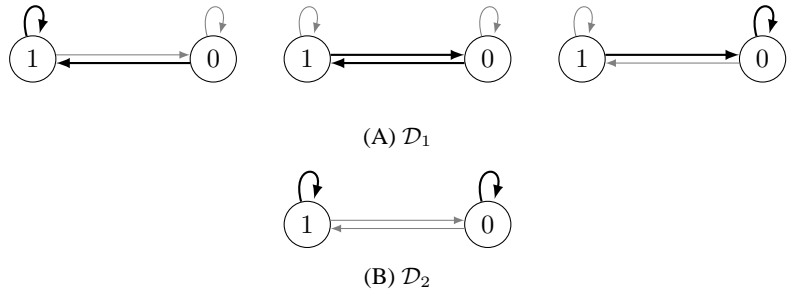
The next two examples investigate the asymptotic independence of length two
blocks of $0$-$1$-sequences.
\begin{example}[$10$- and $11$-blocks]\label{example:blocks}
  The two transducers in Figure~\ref{fig:11-indep-2} count the number of $10$- and
  $11$-blocks in $0$-$1$-sequences. After deleting the outputs, both transducers are the
  same. Thus, any non-degenerate probability distribution on the outgoing edges of the states
  gives a Markov chain with two output functions $k_{10}$ (for the $10$-blocks) and
  $k_{11}$ (for the $11$-blocks). 

Because of the two loops and the cycle $0\to 1\to0$,
Theorem~\ref{theorem:sing-matrix} implies that the number of $10$- and
$11$-blocks is asymptotically normally distributed.

The next question is: For which choices of probability distributions is the number
of $10$- and $11$-blocks asymptotically independent?
All functional digraphs with one or two components are given in
Figure~\ref{fig:func-digraph-2}. Using Theorem~\ref{thm:comb-2}, we obtain the
following system of equations for the values of the probabilities such that the numbers
of $11$-blocks and $10$-blocks are asymptotically independent:
first by definition
\begin{align*}
    1&=p_{0\to 0}+p_{0\to 1},\\
    1&=p_{1\to 0}+p_{1\to1},
  \end{align*}
then by \eqref{eq:expected-value-indep}
  \begin{align*}
    e_{10}&=\frac{p_{0\to 1}p_{1\to 0}}{p_{0\to1}p_{1\to 1}+2p_{0\to
        1}p_{1\to0}+p_{0\to0}p_{1\to 0}},\\
    e_{11}&=\frac{p_{0\to1}p_{1\to 1}}{p_{0\to1}p_{1\to 1}+2p_{0\to
        1}p_{1\to0}+p_{0\to0}p_{1\to 0}},
  \end{align*}
and finally for the independence
  \begin{align*}
p_{0\to1}&p_{1\to1}(-e_{10})(1-e_{11})+p_{0\to1}p_{1\to0}(1-2e_{10})(-2e_{11})+p_{0\to
      0}p_{1\to0}(-e_{10})(-e_{11})\\
    &=p_{0\to0}p_{1\to1}(-e_{10})(-e_{11})+
    p_{0\to0}p_{1\to1}(-e_{10})(1-e_{11}).
\end{align*}
This system has non-trivial real solutions, i.e. solutions where all probabilities
are non-zero, with
\begin{equation*}
 p_{0\to0}= -\frac12 p_{1\to1}+2 - \frac12\sqrt{p_{1\to1}^2 - 8p_{1\to1} + 8}
\end{equation*}
for all $0<p_{1\to1}<1$. Then we have $2-\sqrt{2}<p_{0\to 0}<1$.

Thus, for these transition probabilities, the number of $10$-blocks
and the number of $11$-blocks are asymptotically independent.

One such example of a non-trivial solution is $p_{1\to1}=p_{1\to 0}=0.5$,
$p_{0\to0}\approx 0.7192$ and $p_{0\to1}\approx 0.2808$. Note that for the
symmetric distributions $p_{0\to0}=p_{0\to 1}=p_{1\to1}=p_{1\to 0}=0.5$, we
obtain asymptotic dependence of the number of $10$- and $11$-blocks.
\end{example}

\begin{figure}
  \centering
  \begin{subfigure}[b]{0.45\textwidth}\centering
    \begin{tikzpicture}[auto, initial text=, >=latex, accepting text=,
      accepting/.style=accepting by arrow, every state/.style={minimum
        size=1.3em}]
      \node[state] (v0) at (3.000000, -2.000000) {$0$}; 
      \node[state] (v1) at (3.000000, 2.000000) {$1$}; 
      \node[state, initial] (v2) at (0.000000, 0.000000) {}; 
      \path[->] (v0.0.00) edge node[rotate=0.00, anchor=south] {} ++(0.00:2ex); 
      \path[->] (v1.0.00) edge node[rotate=0.00, anchor=south] {} ++(0.00:2ex); 
      \path[->] (v0.100.00) edge node[rotate=90.00, anchor=south] {$1\mid 0$} (v1.260.00); 
      \path[->] (v0) edge[loop below] node {$0\mid 1$} (); 
      \path[->] (v2) edge node[rotate=33.69, anchor=south] {$1\mid 0$} (v1); 
      \path[->] (v1) edge[loop above] node {$1\mid 0$} (); 
      \path[->] (v2) edge node[rotate=-33.69, anchor=south] {$0\mid 0$} (v0); 
      \path[->] (v1.-80.00) edge node[rotate=90.00, anchor=north] {$0\mid 0$} (v0.80.00);
    \end{tikzpicture}
    \caption{$00$-blocks}
  \end{subfigure}
  \begin{subfigure}[b]{0.45\textwidth}\centering
    \begin{tikzpicture}[auto, initial text=, >=latex, accepting text=,
      accepting/.style=accepting by arrow, every state/.style={minimum
        size=1.3em}]
      \node[state] (v0) at (3.000000, -2.000000) {$0$}; 
      \node[state] (v1) at (3.000000, 2.000000) {$1$}; 
      \node[state, initial] (v2) at (0.000000, 0.000000) {}; 
      \path[->] (v0.0.00) edge node[rotate=0.00, anchor=south] {} ++(0.00:2ex); 
      \path[->] (v1.0.00) edge node[rotate=0.00, anchor=south] {} ++(0.00:2ex); 
      \path[->] (v0.100.00) edge node[rotate=90.00, anchor=south] {$1\mid 0$} (v1.260.00); 
      \path[->] (v0) edge[loop below] node {$0\mid 0$} (); 
      \path[->] (v2) edge node[rotate=33.69, anchor=south] {$1\mid 0$} (v1); 
      \path[->] (v1) edge[loop above] node {$1\mid 1$} (); 
      \path[->] (v2) edge node[rotate=-33.69, anchor=south] {$0\mid 0$} (v0); 
      \path[->] (v1.-80.00) edge node[rotate=90.00, anchor=north] {$0\mid 0$} (v0.80.00);
    \end{tikzpicture}
    \caption{$11$-blocks}
  \end{subfigure}
  \caption{Transducers to compute the number of $00$- and $11$-blocks.}
  \label{fig:00-11}
\end{figure}
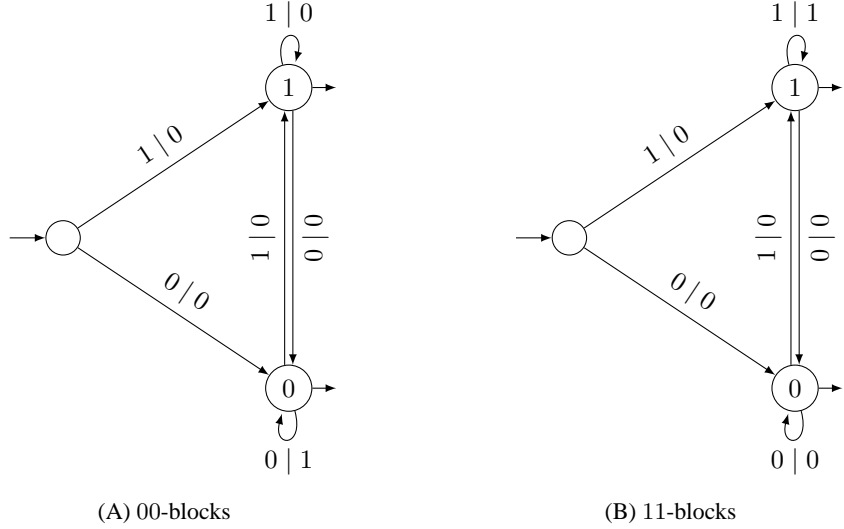
\begin{example}[$00$- and $11$-blocks]\label{example:00-11}
  The two transducers in Figure~\ref{fig:00-11} count the number of $00$- and
  $11$-blocks in $0$-$1$-sequences. They have the same underlying graph and the same input
  labels. Thus, choosing any non-degenerate probability distribution of the
  outgoing edges of the states yields a Markov chain with two
  output functions. 

  Because of the two loops and the cycle $0\to 1\to0$,
Theorem~\ref{theorem:sing-matrix} implies that the number of $00$- and
$11$-blocks is asymptotically normally distributed.

The next question is: For which choices of probability distributions is the number
of $00$- and $11$-blocks asymptotically independent? 
  The functional digraphs of the final component are the same as in
  Example~\ref{example:blocks}, see again Figure~\ref{fig:func-digraph-2}.
  By Theorem~\ref{thm:comb-2}, the system of equations for the transition probabilities
  $p_{e}$ such that the two output functions are asymptotically independent
  are:
  first by definition
  \begin{align*}
      1&=p_{0\to 0}+p_{0\to 1},\\
      1&=p_{1\to 0}+p_{1\to1},
    \end{align*}
    then by \eqref{eq:expected-value-indep}
    \begin{align*}
e_{00}&=\frac{p_{0\to 0}p_{1\to 0}}{p_{0\to1}p_{1\to 1}+2p_{0\to
          1}p_{1\to0}+p_{0\to0}p_{1\to 0}},\\
      e_{11}&=\frac{p_{0\to 1}p_{1\to 1}}{p_{0\to1}p_{1\to 1}+2p_{0\to
          1}p_{1\to0}+p_{0\to0}p_{1\to 0}},
    \end{align*}
and finally for the independence
    \begin{align*}
p_{0\to1}&p_{1\to1}(-e_{00})(1-e_{11})+p_{0\to1}p_{1\to0}(-2e_{00})(-2e_{11})+p_{0\to
        0}p_{1\to0}(1-e_{00})(-e_{11})\\
      &=p_{0\to0}p_{1\to1}(1-e_{00})(1-e_{11})+
      p_{0\to0}p_{1\to1}(-e_{00})(-e_{11}).
  \end{align*}
  These equations have no solution with $0<p_{e}<1$ for all transitions
  $e$. Thus, the numbers of $00$- and $11$-blocks are asymptotically
  dependent for all choices of the input distributions, as expected.
\end{example}

\section{Proofs}\label{sec:proofs}
In this section, we prove the results from
Section~\ref{sec:results}. Most of the
proofs follow along the same ideas as in
\cite{Heuberger-Kropf-Wagner:2014:combin-charac}.
The main differences are that one
has to replace ``complete transducer'' by ``Markov chain'' and the input sum by the
output sum $K_{n}^{(1)}$.

We first prove Theorem~\ref{thm:comb-2} with the help of two lemmas. For one
of these lemmas, we use a version
of the Matrix-Tree Theorem for weighted directed forests proved in
\cite{Chaiken:1982:matrixtree,Moon:1994:matrixtree}. At the
end of this section, we prove Theorems~\ref{theorem:bounded-variance} and~\ref{theorem:sing-matrix}.

\begin{definition}Let $A$, $B\subseteq\{1,\ldots, N\}$. Let $\mathcal F_{A,B}$ be
  the set of all forests
  which are spanning subgraphs of the final component of the Markov chain $\M$ with $|A|$ trees such that
every tree is rooted at some vertex $a\in A$ and contains exactly one
vertex $b\in B$.

Let $A=\{i_{1},\ldots, i_{n}\}$ and $B=\{j_{1},\ldots,
  j_{n}\}$ with $i_{1}<\cdots<i_{n}$ and $j_{1}<\cdots<j_{n}$. For $F\in\mathcal F_{A,B}$, we define a function $g\colon B\to A$ by $g(j)=i$ if $j$ is in the tree of
$F$ which is rooted in vertex $i$. We further define the function
$h\colon A\to B$ by $h(i_{k})=j_{k}$ for $k=1,\ldots,n$. The composition
$g\circ h\colon A\to A$ is a permutation of $A$. We define $\sign F=\sign
g\circ h$.
\end{definition}

If $|A|\neq|B|$, then $\mathcal F_{A,B}=\emptyset$. If $|A|=|B|=1$, then
$\sign F=1$ and $\mathcal F_{A,B}$ consists of all spanning trees
rooted in $a\in A$.

\begin{theorem*}[All-Minors-Matrix-Tree Theorem~\cite{Chaiken:1982:matrixtree,Moon:1994:matrixtree}]
For  a directed, weighted
graph with loops and multiple edges, let $L=(l_{ij})_{1\leq i,j\leq N}$ be the Laplacian matrix, that is $\sum_{j=1}^{N}l_{ij}=0$ for every $i=1,\ldots,N$ and
 $-l_{ij}$ is the sum of the weights $p_{e}$ of all edges $e$ from $i$ to $j$ for $i\neq j$.
Then, for $|A|=|B|$, the minor $\det L_{A,B}$ satisfies
\begin{equation*}\det L_{A,B}=(-1)^{\sum_{i\in A}i+\sum_{j\in B}j}\sum_{F\in \mathcal
  F_{A,B}}p_{F}\sign F\end{equation*}
where $L_{A,B}$ is the matrix $L$ whose rows with index in $A$ and columns with
index in $B$ are deleted.
\end{theorem*}

The All-Minors-Matrix-Tree Theorem is still valid for $|A|\neq |B|$ if we assume
that the determinant of a non-square matrix is $0$. For notational simplicity,
we use this convention in the rest of this section.

\begin{definition}
  The transition matrix $W(x_{1},\ldots,x_{m})$ of a Markov chain
  with $M$ states and $m$ output
  functions $k_{1}$, \dots, $k_{m}$ is a $M\times M$
  matrix whose $(i,j)$-th entry is 
  \begin{equation*}
    \sum_{e\colon i\rightarrow j}p_{e}x_{1}^{k_{1}(e)}\cdots x_{m}^{k_{m}(e)}
  \end{equation*}
  where $p_{e}$ is the probability of the transition $e$. 

  Let $A(x_{1},\ldots,x_{m})$ be the $N\times N$ transition matrix of the final component of the
  Markov chain. Let the order of the states be such that the
  transition matrix of the whole Markov chain $W(x_{1},\ldots,x_{m})$ has the block structure
  \begin{equation}
    \label{eq:block-structure}
    W(x_{1},\ldots,x_{m})=
    \begin{pmatrix}
      *&*\\
      0&A(x_{1},\ldots,x_{m})
    \end{pmatrix}
  \end{equation}
  where $*$ denotes any matrix. If the Markov chain is strongly connected, the
  matrices $*$ are not present (they have $0$ rows).
\end{definition}

We first use the All-Minors-Matrix-Tree Theorem to connect the derivatives of the characteristic
polynomial of the transition matrix with a sum of weighted digraphs in the
next lemma.
\begin{lemma}\label{lem:comb-connection}
  For $f(x_{1},x_{2},z)=\det(I-zA(x_{1},x_{2}))$, we have
\begin{gather*}\begin{aligned}
    f_{x_{i}}(1,1,1) &=-k_{i}(\mathcal D_{1}), & \!\qquad f_{x_{1}x_{2}}(1,1,1)
    &=(k_{1},k_{2})(\mathcal D_{2})-(k_{1},k_{2})(\mathcal D_{1}), \\
    f_{z}(1,1,1) &=-\mathds1(\mathcal D_{1}), & \!\qquad
    f_{x_{i}z}(1,1,1) &=(k_{i},\mathds1)(\mathcal D_{2})-(k_{i},\mathds1)(\mathcal
    D_{1}),    \end{aligned}\\
    \begin{aligned}
      f_{x_{i}x_{i}}(1,1,1)+f_{x_{i}}(1,1,1)&=(k_{i},k_{i})(\mathcal D_{2})-(k_{i},k_{i})(\mathcal
      D_{1}),\\
      f_{zz}(1,1,1)+f_{z}(1,1,1)&=(\mathds1,\mathds1)(\mathcal D_{2})-(\mathds1,\mathds1)(\mathcal D_{1})\\
    \end{aligned}\end{gather*}
for $i=1$, $2$.
\end{lemma}
This lemma can be proven in the same way as
\cite[Lemma~5.3]{Heuberger-Kropf-Wagner:2014:combin-charac} using the
All-Minors-Matrix-Tree Theorem \cite{Chaiken:1982:matrixtree,Moon:1994:matrixtree}.

The following lemma will be used for $m\geq 2$ output functions
later on.
\begin{lemma}\label{lemma:gf}Let
  $f(x_{1},\ldots,x_{m},z)=\det(I-zA(x_{1},\ldots,x_{m}))$. Then there is a
   unique dominant  root $z=\rho(x_{1},\ldots,x_{m})$ of $f$ in a
  neighborhood of $(1,\ldots,1)$.

  The moment generating function of $(K_{n}^{(1)},\ldots,K_{n}^{(m)})$ has the
  asymptotic expansion
  \begin{equation*}
  \mathbb E(\exp(s_{1}K_{n}^{(1)}+\cdots+s_{m}K^{(m)}_{n}))=e^{u(s_{1},\ldots,s_{m})n+v(s_{1},\ldots,s_{m})}(1+\bigOh(\kappa^{n}))
\end{equation*}
where $\kappa<1$,
\begin{align*}
  u(s_{1},\ldots,s_{m})&=-\log\rho(e^{s_{1}},\ldots,e^{s_{m}}),
\end{align*}
and $v(s_{1},\ldots, s_{m})$ are analytic functions in a small neighborhood of $(0,\ldots,0)$.
\end{lemma}
\begin{proof}
  The moment generating function of $(K_{n}^{(1)},\ldots, K_{n}^{(m)})$ is 
  \begin{equation*}
    \mathbb E(\exp(s_{1}K_{n}^{(1)}+\cdots+s_{m}K_{n}^{(m)}))=[z^{n}]v_{1}^{t}(I-zW(e^{s_{1}},\ldots,e^{s_{m}}))^{-1}v_{2}(e^{s_{1}},\ldots,e^{s_{m}})
  \end{equation*}
  for the initial vector $v_{1}$, and a
  vector $v_{2}(x_{1},\ldots,x_{m})$ encoding all the final information of the states\footnote{This
    information is the final output (see
    Remark~\ref{remark:final-outputs}) and the exit weight (see
    Remark~\ref{remark:exit-weights}) included as $w_{i}x_{1}^{f_{1}(i)}\cdots
    x_{m}^{f_{m}(i)}$ in  the $i$-th coordinate of $v_{2}(x_{1},\ldots,x_{m})$. This does not change the asymptotic behavior (see Remark~\ref{remark:final-output-not-important}).} where we write
  $[z^{n}]b(z)$ for the coefficient of $z^{n}$ in the power series $b$. Because of the
  block structure of the transition matrix $W$ of the whole Markov chain in~\eqref{eq:block-structure}, we
  obtain
  \begin{align*}
    \mathbb
    E(x_{1}^{K_{n}^{(1)}}\cdots x_{m}^{K^{(m)}_{n}})&=[z^{n}]\frac{F_{1}(x_{1},\ldots,x_{m},z)}{\det(I-zW(x_{1},\ldots,x_{m}))}\\
    &=[z^{n}]\frac{F_{1}(x_{1},\ldots,x_{m},z)}{F_{2}(x_{1},\ldots,x_{m},z)f(x_{1},\ldots,
      x_{m},z)}
  \end{align*}
  for ``polynomials'' $F_{1}$ and $F_{2}$ , i.e.\ finite linear combinations of
  $x_{1}^{\alpha_{1}}\cdots x_{m}^{\alpha_{m}}z^{\beta}$ for $\alpha_{i}\in\mathbb R$ and
  $\beta$ a non-negative integer. The function $F_{2}$ corresponds to the
  determinant of the non-final part of the Markov chain.

We obtain the coefficient of $z^{n}$ by singularity analysis  (cf.\ \cite{Flajolet-Sedgewick:ta:analy}):
Since the final component of $\M$ is again a Markov chain, the dominant
singularity of $1/f(1,\ldots, 1,z)$  is $1$ by the theorem of Perron--Frobenius (cf.~\cite{Godsil-Royle:2001:alggraphtheory}). By the aperiodicity of the final
component, this dominant singularity is unique and it is $\rho(1,\ldots,1)=1$. 

Next, we consider the non-final components of the Markov chain using the same arguments as in \cite{Heuberger-Kropf-Wagner:2014:combin-charac}. The
corresponding non-final component $\M_{0}$ is not a Markov chain as the
transition matrix is not stochastic. Let $\M_{0}^{+}$ be the
Markov chain that is obtained from $\M_{0}$ by adding loops with the missing probabilities where
necessary. The
dominant eigenvalue of the transition matrix of $\M_{0}^{+}$ is $1$. As the transition matrices of $\M_{0}$ and $\M_{0}^{+}$ satisfy element-wise inequalities
but are not equal (at $(x_{1},\ldots,x_{m})=(1,\ldots,1)$), the theorem of Perron--Frobenius
(cf.~\cite[Theorem~8.8.1]{Godsil-Royle:2001:alggraphtheory}) implies that the
dominant eigenvalues of $\M_{0}$ have absolute value less than $1$.
Thus, the dominant singularities of $F_{2}(1,\ldots,1,z)^{-1}$ are at $\lvert
z\rvert>1$.

As $A(1,\ldots,1,z)=(1-z)^{-1}$, we obtain $F_{1}(1,\ldots,1)\neq 0$.

Thus, there is a is the
unique, dominant singularity of
\begin{equation*}
  \frac{F_{1}(1,\ldots,1,z)}{F_{2}(1,\ldots,1,z)f(1,\ldots,1,z)},
\end{equation*}
which is $\rho(1,\ldots,1)=1$.
This
also holds for $(x_{1},\ldots,x_{m})$ in a small neighborhood of $(1,\ldots,1)$ by the continuity of
the eigenvalues of the transition matrices. Thus, $\rho(x_{1},\ldots, x_{m})$ is this unique
dominant singularity.

Now, singularity analysis (cf.~\cite{Flajolet-Sedgewick:ta:analy}) implies
the statement of this lemma.
\end{proof}

\begin{remark}\label{remark:final-output-not-important}
  The main term of the asymptotic expansion of the moment generating
  function only depends on $\rho(x_{1},\ldots,x_{m})$ and therefore on
  $f(x_{1},\ldots,x_{m},z)$. It does not depend on the
  ``polynomials'' $F_{1}(x_{1},\ldots,x_{m},z)$ and
  $F_{2}(x_{1},\ldots,x_{m},z)$. Thus, only the final component influences the main
  term. Neither the states in the non-final part of the Markov chain nor the
  final outputs and exit weights
  influence the main term.
\end{remark}

Now, we can use the previous two lemmas to prove
Theorem~\ref{thm:comb-2}.

\begin{proof}[of Theorem~\ref{thm:comb-2}.]
By Lemma~\ref{lemma:gf} for two output functions $k_{1}$ and $k_{2}$, the moment generating function satisfies the conditions of
the Quasi-Power Theorem~\cite[Theorem~5.1]{Heuberger-Kropf-Wagner:2014:combin-charac}, which yields the expected value
\begin{equation*}
  \mathbb E (K_{n}^{(1)}, K^{(2)}_{n})=n\grad u(\bfzero)+\bigOh(1)
\end{equation*}
and the variance
\begin{equation*}
  \mathbb V(K^{(1)}_{n},K^{(2)}_{n})=nH_{u}(\bfzero)+\bigOh(1)
\end{equation*}
with $\grad u(\bfzero)$ and $H_{u}(\bfzero)$ the gradient and the Hessian of
$u$ at $\bfzero$, respectively. Furthermore, we obtain an asymptotic joint normal
distribution of the standardized random vector if the Hessian is not
singular by
\cite[Theorem~3.9]{Heuberger-Kropf-Wagner:2014:combin-charac}. Otherwise, the
limiting random vector is either a pair of degenerate random variables, or a
degenerate and normally distributed one, or a linear transformation thereof. Thus, the random variables $K^{(1)}_{n}$ and $K^{(2)}_{n}$ are asymptotically
independent if and only if the covariance is zero.

By implicit differentiation, we obtain the following formulas for the constants
of the moments in terms of the partial derivatives of $f$:
\begin{align*}
e_{i}&=\frac{f_{x_{i}}}{f_z}\Big\vert_{\bfones},\\
v_{i}&=\frac1{f_{z}^{3}}(f_{x_{i}}^{2}(f_{zz}+f_{z})+f_{z}^{2}(f_{x_{i}x_{i}}+f_{x_{i}})-2f_{x_{i}}f_{z}f_{x_{i}z})\Big\vert_{\bfones},\\
c&=
\frac1{f_{z}^{3}}(f_{x_{1}}f_{x_{2}}(f_{zz}+f_{z})+f_{z}^{2}f_{x_{1}x_{2}}-f_{x_{2}}f_{z}f_{x_{1}z}-f_{x_{1}}f_{z}f_{x_{2}z})\Big\vert_{\bfones}
\end{align*}
for $i=1$, $2$.

Now, Lemma~\ref{lem:comb-connection} implies the results as stated in the
theorem. 
\end{proof}

\begin{proof}[of Theorem~\ref{theorem:bounded-variance}]
This follows by the same arguments as in \cite[Theorem~3.1]{Heuberger-Kropf-Wagner:2014:combin-charac}.
\end{proof}

\begin{proof}[of Corollary~\ref{cor:01output-2}]
This follows by the same arguments as in \cite[Corollary
3.6]{Heuberger-Kropf-Wagner:2014:combin-charac}.
\end{proof}

\begin{proof}[of Theorem~\ref{theorem:sing-matrix}]
  WLOG, we assume that $\Expect K_{n}^{(i)}=\bigOh(1)$ for
  $i=1,\ldots,m$ by subtracting the corresponding constant of the expected value from each
  output function. There exists a unitary matrix $T=(t_{ji})_{1\leq j,i\leq m}$ such that the
  variance-covariance matrix $\Sigma$ can be diagonalized as $T\Sigma
  T^{\top}=D$. The diagonal matrix $D$ is the variance-covariance matrix of
  the linearly transformed random vector $\boldsymbol{Y}_{n}=T\boldsymbol{K}_{n}$.

Then $\Sigma$ is singular if and only if the diagonal matrix
  $D$ is singular. This is equivalent to 
  \begin{equation}\label{eq:var-bounded-2}
    \mathbb V(t_{j1}K_{n}^{(1)}+\cdots+ t_{jm}K_{n}^{(m)})=\bigOh(1)
  \end{equation}
  holds for a $j\in\{1,\ldots,m\}$. Now consider the output function
  $t_{j1}k_{1}+\cdots+t_{jm}k_{m}$. By Theorem~\ref{theorem:bounded-variance},
  \eqref{eq:var-bounded-2}
  is equivalent to
  \begin{equation*}
    t_{j1}k_{1}(C)+\cdots+t_{jm}k_{m}(C)=0
  \end{equation*}
  holding for all cycles of the final component (since the expected value of this
  output function is $\bigOh(1)$).

If we shift back the output function such that the expected value
is no longer bounded, we obtain an additional summand $a_{0}\mathds 1(C)$.

The asymptotic joint normal distribution follows from Lemma~\ref{lemma:gf} and
the multidimensional Quasi-Power Theorem~\cite[Theorem~2.22]{Drmota:2009:random}.
\end{proof}
\bibliographystyle{amsplain}
\bibliography{cheub}

\providecommand{\Submitted}{Submitted} \providecommand{\availableat}{ available
  at } \providecommand{\alsoavailableat}{ also available at }
  \providecommand{\evavailableat}{earlier version available at }
  \providecommand{\toappearin}{To appear in } \providecommand{\toappear}{to
  appear} \providecommand{\inpreparation}{in preparation}
  \providecommand{\doi}[1]{\href{http://dx.doi.org/#1}{\path{doi:#1}}}
  \providecommand{\etc}{\emph{etc.}}\def\cprime{$'$}
\providecommand{\bysame}{\leavevmode\hbox to3em{\hrulefill}\thinspace}
\providecommand{\MR}{\relax\ifhmode\unskip\space\fi MR }
\providecommand{\MRhref}[2]{%
  \href{http://www.ams.org/mathscinet-getitem?mr=#1}{#2}
}
\providecommand{\href}[2]{#2}
\begin{thebibliography}{10}

\bibitem{avanzi:mywnaf}
Roberto Avanzi, \emph{A note on the signed sliding window integer recoding and
  a left-to-right analogue}, {Selected Areas in Cryptography: 11th
  International Workshop, SAC 2004, Waterloo, Canada, August 9-10, 2004,
  Revised Selected Papers} (H.~Handschuh and A.~Hasan, eds.), Lecture Notes in
  Comput. Sci., vol. 3357, Springer-Verlag, Berlin, 2005, pp.~130--143.

\bibitem{Avanzi-Heuberger-Prodinger:2006:scalar-multip-koblit-curves}
Roberto Avanzi, Clemens Heuberger, and Helmut Prodinger, \emph{Scalar
  multiplication on {K}oblitz curves. {U}sing the {F}robenius endomorphism and
  its combination with point halving: {E}xtensions and mathematical analysis},
  Algorithmica \textbf{46} (2006), 249--270. \MR{2291956 (2008a:94180)}

\bibitem{Bender-Kochman:1993:distr-subwor}
Edward~A. Bender and Fred Kochman, \emph{The distribution of subword counts is
  usually normal}, European J. Combin. \textbf{14} (1993), no.~4, 265--275.

\bibitem{Berthe-Rigo:2010:combin}
Val\'{e}rie Berth\'{e} and Michel Rigo (eds.), \emph{Combinatorics, automata
  and number theory}, Encyclopedia Math. Appl., vol. 135, Cambridge University
  Press, Cambridge, 2010.

\bibitem{Chaiken:1982:matrixtree}
Seth Chaiken, \emph{A combinatorial proof of the all minors matrix tree
  theorem}, SIAM J. Alg. Disc. Meth. \textbf{3} (1982), no.~3, 319--329.

\bibitem{Drmota:2009:random}
Michael Drmota, \emph{Random trees}, SpringerWienNewYork, 2009. \MR{2484382
  (2010i:05003)}

\bibitem{Flajolet-Sedgewick:ta:analy}
Philippe Flajolet and Robert Sedgewick, \emph{Analytic combinatorics},
  Cambridge University Press, Cambridge, 2009.

\bibitem{Flajolet-Szpankowski-Vallee:2006:hidden-word-statis}
Philippe Flajolet, Wojciech Szpankowski, and Brigitte Vall{\'e}e, \emph{Hidden
  word statistics}, J.\ ACM \textbf{53} (2006), no.~1, 147--183.

\bibitem{Godsil-Royle:2001:alggraphtheory}
Chris~D. Godsil and Gordon Royle, \emph{Algebraic graph theory}, Graduate texts
  in mathematics, vol. 207, Springer Verlag (New York), 2001.

\bibitem{Goldwurm-Radicioni:2007:averag-value}
Massimiliano Goldwurm and Roberto Radicioni, \emph{Average value and variance
  of pattern statistics in rational models}, Implementation and Application of
  Automata (Jan Holub and Jan \v{Z}\v{d}\'{a}rek, eds.), Lecture Notes in
  Comput. Sci., vol. 4783, Springer Berlin Heidelberg, 2007, pp.~62--72.

\bibitem{Grabner-Heuberger-Prodinger:2004:distr-results-pairs}
Peter~J. Grabner, Clemens Heuberger, and Helmut Prodinger, \emph{Distribution
  results for low-weight binary representations for pairs of integers},
  Theoret. Comput. Sci. \textbf{319} (2004), 307--331. \MR{2074958
  (2005h:11018)}

\bibitem{Grabner-Heuberger-Prodinger-Thuswaldner:2005:analy-linear}
Peter~J. Grabner, Clemens Heuberger, Helmut Prodinger, and J{\"o}rg
  Thuswaldner, \emph{Analysis of linear combination algorithms in
  cryptography}, ACM Trans. Algorithms \textbf{1} (2005), 123--142. \MR{2163134
  (2006j:65425)}

\bibitem{Grabner-Thuswaldner:2000:sum-of-digits-negative}
Peter~J. Grabner and J{\"o}rg~M. Thuswaldner, \emph{On the sum of digits
  function for number systems with negative bases}, Ramanujan J. \textbf{4}
  (2000), no.~2, 201--220. \MR{1782199 (2001m:11015)}

\bibitem{Heigl-Heuberger:2012:analy-digit}
Florian Heigl and Clemens Heuberger, \emph{Analysis of digital expansions of
  minimal weight}, 23rd Intern. Meeting on Probabilistic, Combinatorial, and
  Asymptotic Methods for the Analysis of Algorithms (AofA'12), DMTCS
  Proceedings, 2012, pp.~399--411. \MR{2957346}

\bibitem{Heuberger-Kropf:2013:analy}
Clemens Heuberger and Sara Kropf, \emph{Analysis of the binary asymmetric joint
  sparse form}, Combin. Probab. Comput. \textbf{23} (2014), 1087--1113.
  \MR{3265839}

\bibitem{Heuberger-Kropf-Prodinger:2015:output}
Clemens Heuberger, Sara Kropf, and Helmut Prodinger, \emph{Output sum of
  transducers: Limiting distribution and periodic fluctuation}, Electron. J.
  Combin. \textbf{22} (2015), no.~2, 1--53. \MR{3359922}

\bibitem{Heuberger-Kropf-Prodinger:ta:analy-carries}
\bysame, \emph{Analysis of carries in signed digit expansions}, Monatsh. Math.
  (2016), published online first, doi:10.1007/s00605-016-0917-x.

\bibitem{Heuberger-Kropf-Wagner:2014:combin-charac}
Clemens Heuberger, Sara Kropf, and Stephan Wagner, \emph{Variances and
  covariances in the central limit theorem for the output of a transducer},
  European J. Combin. \textbf{49} (2015), 167--187. \MR{3349532}

\bibitem{Heuberger-Prodinger:2006:analy-alter}
Clemens Heuberger and Helmut Prodinger, \emph{Analysis of alternative digit
  sets for nonadjacent representations}, Monatsh. Math. \textbf{147} (2006),
  219--248. \MR{2215565 (2007g:11010)}

\bibitem{Moon:1994:matrixtree}
John~W. Moon, \emph{Some determinant expansions and the matrix-tree theorem},
  Discrete Math. \textbf{124} (1994), 163--171.

\bibitem{muirstinson:minimality}
James~A. Muir and Douglas~R. Stinson, \emph{Minimality and other properties of
  the width-{$w$} nonadjacent form}, Math. Comp. \textbf{75} (2006), 369--384.

\bibitem{Nicodeme-Salvy-Flajolet:2002:motif}
Pierre Nicod\`{e}me, Bruno Salvy, and Philippe Flajolet, \emph{Motif
  statistics}, Theoret. Comput. Sci. \textbf{287} (2002), no.~2, 593--617.

\bibitem{Parry:1964:intrin-markov}
William Parry, \emph{Intrinsic {M}arkov chains}, Trans. Amer. Math. Soc.
  \textbf{112} (1964), 55--66.

\bibitem{Shannon:1948:mathem-theor-commun}
Claude~E. Shannon, \emph{A mathematical theory of communication}, Bell System
  Tech. J. \textbf{27} (1948), 379--423.

\end{thebibliography}

\end{document}